\documentclass[reqno,11pt]{amsart}
\usepackage{amsmath, latexsym, amsfonts, amssymb, amsthm, amscd}
\usepackage[utf8]{inputenc}
\usepackage{graphics,epsf,psfrag,epsfig}

\numberwithin{equation}{section}

\newtheorem{theorem}{Theorem}[section]
\newtheorem{lemma}[theorem]{Lemma}
\newtheorem{proposition}[theorem]{Proposition}

\newtheorem{rem}[theorem]{Remark}

\DeclareMathOperator{\var}{\mathrm{Var}}
\DeclareMathOperator{\cov}{\mathrm{Cov}}

\renewcommand{\ge}{\geq}
\renewcommand{\le}{\leq}
\newcommand{\ind}{\mathbf{1}}

\renewcommand{\tilde}{\widetilde}

\newcommand{\cA}{{\ensuremath{\mathcal A}} }

\DeclareMathSymbol{\leqslant}{\mathalpha}{AMSa}{"36} 
\DeclareMathSymbol{\geqslant}{\mathalpha}{AMSa}{"3E} 
\DeclareMathSymbol{\eset}{\mathalpha}{AMSb}{"3F}     
\renewcommand{\leq}{\;\leqslant\;}                   
\renewcommand{\geq}{\;\geqslant\;}                   
\newcommand{\dd}{\,\text{\rm d}}             

\newcommand{\sumtwo}[2]{\sum_{\substack{#1 \\ #2}}} 

\newcommand{\bbE}{{\ensuremath{\mathbb E}} }

\newcommand{\bbN}{{\ensuremath{\mathbb N}} }

\newcommand{\bbP}{{\ensuremath{\mathbb P}} }

\newcommand{\bbR}{{\ensuremath{\mathbb R}} }

\newcommand{\bbZ}{{\ensuremath{\mathbb Z}} }

\newcommand{\gb}{\beta}
\newcommand{\gd}{\delta}

\newcommand{\go}{\omega}

\newcommand{\gl}{\lambda}

\makeatletter
\def\captionfont@{\footnotesize}
\def\captionheadfont@{\scshape}

\long\def\@makecaption#1#2{%
  \vspace{2mm}
  \setbox\@tempboxa\vbox{\color@setgroup
    \advance\hsize-6pc\noindent
    \captionfont@\captionheadfont@#1\@xp\@ifnotempty\@xp
        {\@cdr#2\@nil}{.\captionfont@\upshape\enspace#2}%
    \unskip\kern-6pc\par
    \global\setbox\@ne\lastbox\color@endgroup}%
  \ifhbox\@ne 
    \setbox\@ne\hbox{\unhbox\@ne\unskip\unskip\unpenalty\unkern}%
  \fi
  \ifdim\wd\@tempboxa=\z@ 
    \setbox\@ne\hbox to\columnwidth{\hss\kern-6pc\box\@ne\hss}%
  \else 
    \setbox\@ne\vbox{\unvbox\@tempboxa\parskip\z@skip
        \noindent\unhbox\@ne\advance\hsize-6pc\par}%
\fi
  \ifnum\@tempcnta<64 
    \addvspace\abovecaptionskip
    \moveright 3pc\box\@ne
  \else 
    \moveright 3pc\box\@ne
    \nobreak
    \vskip\belowcaptionskip
  \fi
\relax
}
\makeatother
\def\writefig#1 #2 #3 {\rlap{\kern #1 truecm
\raise #2 truecm \hbox{#3}}}

\title[Connective constant of 2D dilute lattice]
{Non-coincidence of quenched and annealed connective constants on the supercritical
planar percolation cluster}

\address{CEREMADE, Place du Mar\'echal De Lattre De Tassigny
75775 PARIS CEDEX 16 - FRANCE}
\email{lacoin@ceremade.dauphine.fr}
\author{Hubert Lacoin}

\begin{document}

\begin{abstract}
In this paper, we study the abundance of self-avoiding paths of a given length on a supercritical percolation cluster on $\bbZ^d$.
More precisely, we count $Z_N$, the number of self-avoiding paths of length $N$ on the infinite cluster
starting from the origin (which we condition to be in the cluster). We are interested in
estimating the upper growth rate of $Z_N$, $\limsup_{N\to \infty} Z_N^{1/N}$, which
we call the connective constant of the dilute lattice.
After proving that this connective constant is a.s.\ non-random,
we focus on the two-dimensional case and
show that for every percolation parameter $p\in (1/2,1)$, almost surely, 
$Z_N$ grows exponentially slower than its expected value.  In other words, we prove that 
$\limsup_{N\to \infty} (Z_N)^{1/N} <\lim_{N\to \infty} \bbE[Z_N]^{1/N}$,
where the expectation is 
taken with respect to the percolation process.
This result can be considered as a first mathematical attempt to understand the influence of disorder 
for self-avoiding walks on a (quenched) dilute lattice.
Our method, which combines change of measure and coarse graining arguments, 
does not rely on the specifics of percolation on $\bbZ^2$, so our result can 
be extended to a large family of two-dimensional models including general self-avoiding walks in a random environment.
\end{abstract}
\maketitle

\section{model and result}

\subsection{Introduction}

We are interested in percolation on the $\bbZ^d$ grid ($d\ge 2$) with its usual lattice structure.
We delete each edge with probability $1-p$ and investigate the connectivity properties of the resulting 
network (or rather of its unique infinite connected component when it exists).
More precisely, we want to study the asymptotic growth of the number of \textit{self-avoiding paths} of length $N$
starting from a given (typical) point on the dilute lattice.
A self-avoiding path is a lattice path  that does not visit the same 
vertex twice.

\medskip

Comparing the number of self-avoiding paths with its expected value gives some heuristic information concerning 
the influence of a quenched edge dilution on the trajectorial behavior of the self-avoiding walk.

\subsection{The self-avoiding walk}
Let us first recall some facts about the self-avoiding walk on a regular lattice (we focus on $\bbZ^d$ for the sake of simplicity).
Set 
\begin{equation}
\mathcal S_N:=
\{ \text{ self-avoiding paths $(S_n)_{n\in[0,N]}$  on $\bbZ^d$ of length $N$ starting at $0$ }\}
\end{equation}
and $s_N:=|\mathcal S_N|$.
As $s_N$ is a submultiplicative function, the limit 
\begin{equation}
 \lim_{n\to \infty} (s_N)^{1/N}:=\mu_d
\end{equation}
exists.

\medskip

The constant $\mu_d$ is called the connective constant of the network.
It is not expected to take any remarkable value as far as $\bbZ^d$ is concerned (on the two dimensional honeycomb 
lattice on the contrary,
it has been conjectured for a long time and has been recently proved that $\mu=\sqrt{2+\sqrt 2}$, see \cite{cf:DS}).

\medskip

The self-avoiding walk of length $N$ is a stochastic process whose law is given by the uniform probability measure on
$\mathcal S_N$. It has been introduced as a model for polymers by Flory \cite{cf:Flo}.
Theoretical physicists have then been interested in describing typical behavior of the walk for large $N$,  to understand
whether it differs from that of the simple random walk and why.
Their answer to this question depends on the dimension:
\begin{itemize}
 \item [(i)] When $d>4$, the self-avoiding constraint is a local one. Indeed, around a typical point of a simple random walk's trajectory, the 
 past and the future intersect finitely many times, at a finite distance.
 For this reason, the self-avoiding walk in dimension larger than $4$ scales like Brownian Motion.
 The case $d=4$ which corresponds to the critical dimension, should be similar but with logarithmic corrections.
 \item [(ii)] When $d<4$, the self-avoiding constraint acts also on a large scale and modifies the macroscopic structure of the walk.
 In particular, it forces the walk to go further: the end to end distance $|S_N|$ is believed to scale like 
 $N^\nu$, where $\nu=3/4$ for $d=2$ and $\nu\simeq 0.59$ for $d=3$. 
\end{itemize}

On the mathematical side, the picture is much less complete.
Above the critical dimension, when $d>4$, the use of the lace expansion by Brydges and Spencer \cite{cf:BSp}
allowed to make  the physicists prediction rigorous, but when $d<4$, very few things are known rigorously
(for a complete introduction to the subject and
a list of the conjecture see the first chapter of  \cite{cf:MSlade}, or \cite{cf:Slade} for a more recent survey).
Note that recently, Duminil-Copin and Hammond proved that the self-avoiding walk is non-ballistic in every dimension
\cite{cf:DH}.

\subsection{Percolation on $\bbZ^d$}

Let $\go$ be the edge dilution (or percolation) process defined on the set of the edges of $\bbZ^d$ as follows:
\begin{itemize}
 \item $(\go(e))_{e\in E_d}$ is a field of IID $\{0,1\}$ Bernoulli variables with law $\bbP_p$
 satisfying $\bbP_p(\go(e)=1)=p$. 
 \item Every edge $e$ such that $\go(e)=1$ is declared open or present whereas the others are deleted (or closed).
\end{itemize}
A set of edges is declared open if all the edges in it are open ; a self-avoiding path $S$ is declared open if all the
edges in the path are open (we will use the informal notation $e\in S$ to say that $e=(S_n,S_{n+1})$ for some $n$). 

\medskip

The nature of the new lattice obtained after deleting edges depends  crucially on the value of $p$. There exists a 
constant $p_c(d)$
called the percolation threshold such that the dilute lattice contains 
a unique infinite connected component if $p\in(p_c,1]$ 
(in addition to countably many finite connected components), 
and none if $p<p_c$. It is also known that $p_c(2)=1/2$
(see e.g.\ \cite{cf:G} for a complete introduction to percolation).

\subsection{The quenched connective constant for the percolation cluster}

In what follows we consider exclusively the supercritical percolation regime where $p>p_c$. We let $\mathcal C$ denote  the supercritical percolation cluster
(the unique infinite connected component).

\medskip

Set $Z_N$ to be the number of open self-avoiding paths of length $N$ starting from the origin:

\begin{equation}
 Z_N:=\sum_{S\in \mathcal S_N} \ind_{\{\text{$S$ is open}\}}.
\end{equation}
Similarly, one can define $Z_N(x)$ by considering paths that starts from $x$ instead of paths starting from the origin.
One has trivially 
\begin{equation}
 Z_N(x)=0 \text{ for all sufficiently large } N \Leftrightarrow x \notin \mathcal C.
\end{equation}
Our aim is to study the asymptotic behavior of $Z_N(x)$ when $x$ belongs to $\mathcal C$.
One can easily compute its expectation: for $x$ in $\bbZ^d$ we have

\begin{equation}\label{coco}
\bbE_p\left[Z_N(x)\right]=\bbE_p[Z_N]=\sum_{S\in \mathcal S_N}\bbP_p\left[\text{$S$ is open}\right]= p^N s_N,
\end{equation}
and this gives an upper bound on the possible growth rate of $Z_N(x)$
(we cannot compute  the expectation conditioned on $x\in \mathcal C$ exactly, but the reader can check that it has the same order of magnitude).

\medskip

Now we define two versions of the connective 
constant $\mu_d$ for the infinite percolation cluster
$\mathcal C$. 
 We call
\begin{equation}
\lim_{N\to \infty} \bbE_p \left[Z_N\right]^{\frac1N}= p\mu_d
\end{equation}
the \textit{annealed} connective constant. To define the \textit{quenched} one, we 
prove the following result is which valid in any dimension.

\begin{proposition}\label{taex}
For every $x\in \mathcal C$, the limit
\begin{equation}\label{limsupp}
 \mu_d(p):=\limsup_{N\to \infty} (Z_N(x))^{\frac 1 N},
\end{equation}
does not depend on $x$ and is non-random. We call it the quenched connective constant.

It satisfies the inequality
\begin{equation}\label{annqu}
 \mu_d(p)\le p \mu_d(1).
\end{equation}

\medskip

Moreover, the ratio $$\mu_d(p)/p\mu_d(1)$$ between the quenched and the  annealed connective constants is a 
non-decreasing function of $p$ on $(p_c,1]$.
\end{proposition}

\begin{rem}\rm
 We believe that
\begin{equation}
  \lim_{N\to \infty}  (Z_N(x))^{\frac 1 N}
\end{equation}
exists, but the best we can do here is to state this as a conjecture.
\end{rem}

We are interested in knowing whether or not the inequality \eqref{annqu} is sharp.
The reason for this interest is that at a heuristic level, the ratio $Z_N/\bbE[Z_N]$ conveys some information
on the trajectorial behavior of the self-avoiding walk on the dilute lattice. The self-avoiding walk of length $N$ on the dilute lattice 
is the stochastic process whose law is given by the uniform probability measure on the random set
\begin{equation}
 \mathcal S_N(\go):=\{S\in \mathcal S_N \ | \ S \text{ is open for $\go$ }\}.
\end{equation}
Note that this definition makes sense for all $N$ only if $0\in \mathcal C$.

In analogy with what happens for directed polymers in a random environment, we believe that:
\begin{itemize}
 \item [(i)] If $Z_N/\bbE[Z_N]$ is typically of order $1$ then, 
 the self-avoiding walk on the dilute lattice has a similar behavior  to the walk on the full lattice.
 \item [(ii)] If $Z_N/\bbE[Z_N]$ decays exponentially fast, then disorder changes the behavior of the trajectories.
 It induces localization of trajectories (they concentrate in the regions where $\go$ is more favorable), and possibly
 stretches them, making the end-to-end distance $|S_N|$ larger.  
\end{itemize}
Using some of the techniques that have been used for directed polymers could bring these statements onto more rigorous ground 
(see \cite{cf:CY} for an analogy with case $(i)$ and \cite{cf:CH} for an analogy with the localization part of $(ii)$). 
Of course, saying something rigorous about the end-to-end distance for the disordered model is quite hopeless as it is already a difficult open question 
for the homogeneous model.

\begin{rem}\rm \label{fastrans}
As the ratio $\mu_d(p)/p\mu_d(1)$ is non-decreasing, there exists a unique $\bar p_c\in[p_c,1]$ such that
$\mu_d(p)<p\mu_d(1)$ for $p<(p_c,\bar p_c)$ and $\mu_d(p)=p\mu_d(1)$ for  $p\in(\bar p_c,1]$ (one of these intervals  possibly being empty).

\medskip

If $\bar p_c\in (p_c,1)$ then the function $p\mapsto \mu_d(p)$ cannot be analytic around  $\bar p_c$ so that $\bar p_c$ delimits a phase transition in the usual sense of the term, between 
what we call a localized or strong-disorder phase $(p_c,\bar p_c)$ and a weak-disorder phase phase $(\bar p_c,1]$.
We have no evidence that $p\mapsto \mu_d(p)$ is analytic or even continuous on $(p_c,\bar p_c)$.

\end{rem}

\subsection{Main result}

The main result of this paper is that the quenched connective constant is strictly smaller than the
annealed one for the model on $\bbZ^2$, suggesting localization of the trajectories.

\begin{theorem}\label{superteo}
For every $p\in (p_c(2),1)$
\begin{equation}\label{trucpricn}
 \mu_2(p)<p \mu_2(1),
\end{equation}
meaning that $Z_N/\bbE_p\left[Z_N\right]$ 
tends to zero exponentially fast.
Moreover, the function 
$$p\mapsto \frac{\mu_2(p)}{p\mu_2(1)}$$
is strictly increasing on $(p_c(2),1]$.
\end{theorem}

Although the proof allows one to extract an explicit upper bound for $\mu_2(p)-p\mu_2(1)$, which gets exponentially small 
when $p$ approaches one,
we believe it to be  far from optimal when edge dilution is small. Indeed, if $|S_N|$ scales like $N^{\nu}$ with $\nu<1$, then the argument of 
the proof in \cite[Section 3]{cf:L} give at a heuristic level that $p\mu_2(1)-\mu_2(p)$ is at least of order 
$(1-p)^{\frac{1}{2(1-\nu)}}$
(which is to be compared with
the bound in \eqref{irrelvt}).

\subsection{Comparison with predictions in the physics literature}

Although the physics literature concerning the self-avoiding walk on a dilute lattice is quite rich (for the first paper on the subject see \cite{cf:CK}), 
it is difficult to extract  a solid conjecture on the value of $\mu_d(p)$ from the variety of contributions.

\medskip

The first reason is that most of the studies focus on the trajectorial behavior, and only 
 marginal attention is given to the partition function $Z_N$
(in that respect,
\cite{cf:BKC} is a noticeable exception with explicit
focus on $\mu_d(p)$). 

\medskip

The second reason is that while it is often not stated explicitly, it seems that most of the papers
from the eighties seem to be focused on the {\sl annealed} models of self-avoiding walk on a percolation cluster, which is mathematically trivial
(see for instance the sequence of equations to compute the mean square of the end-to-end distance in \cite{cf:H2}).
\medskip

Because of this last remark, we do not feel that our result contradicts the many papers (see e.g.\ \cite{cf:Kr, cf:H2, cf:HM}) which
predict that edge-dilution does not change the walk's behavior. The few numerical studies available concerning the connective constant are not very informative either:
both  \cite[Table 1]{cf:CR} and \cite[Figure 3]{cf:BKC} give values for $\mu_2(p)$ that violates the annealed bound $\mu_d(p)\le p \mu_d(1)$.

\medskip

In \cite{cf:LDM}, the authors clearly state that they study the quenched problem, and make a number of predictions partially based on a
renormalization group study performed on hierarchical lattices:
\begin{itemize}
 \item When $d=2,3$, there is no weak disorder/strong disorder phase transition, and an arbitrary small dilution changes the properties of the self-avoiding walk.
 \item When $d>4$, a small edge-dilution does not change the trajectorial property, and there is a phase transition from a weak disorder phase 
 to a strong disorder phase when $p$ varies.
\end{itemize}
Furthermore, they give an explicit formula linking the typical fluctuations of $\log Z_N$ around its mean with the end-to-end exponent $\nu$ in the strong disorder phase.
This prediction agrees with the earlier one present in \cite{cf:CK}, where the Harris Criterion (from \cite{cf:Hcrit}) is used to decide on  disorder relevance.

\medskip

To our understanding, our result partially confirms the prediction of \cite{cf:LDM} in dimension $2$.
We would translate the higher dimension predictions in terms of the quenched connective constant as follows:
\begin{itemize}
 \item  For $d=3$ $\mu_d(p)<p\mu_d(1)$ for all $p<1$.
 \item  When $d>4$, we have $\bar p_c(d)\in (p_c(d),1)$ where $\bar p_c(d)$ is defined in Remark \ref{fastrans}.
 \end{itemize}
 
 This is quite similar to what happens for directed percolation (see \cite{cf:HDP}) where
for $d\le 2$ the number of open directed paths is much smaller than its expected value for every $p$, while 
when $d\ge 3$ a weak disorder phase exists. 
 
 \medskip
 
 In Section \ref{fredy} we prove a result that supports this conjecture, namely that if the volume exponent in dimension $3$ is smaller than $2/3$, then 
 $Z_N$ is typically much smaller than its expectation (but not necessarily exponentially smaller).

\medskip

Predicting anything about the critical dimension $d=4$ is trickier, and all of this remains at a very speculative level.
Also, the existence of a weak disorder phase is a much more challenging question than for directed percolation for which it can be proved with a two-line computation (by computing the second moment of $Z_N$).

\medskip

Note that the relevant critical dimension for the problem we are interested in should be the one of the random walk and not the one of the percolation process: $d=4$ for the self-avoiding walk, and $d=2$ 
for the simple random walk in the oriented model.

\begin{rem} \rm After the first draft of this work appeared, we have been able to make a step further in confirming the physicists' prediction by proving that   in large enough dimensions $\bar p_c(d)>p_c$, i.e.\ that 
the strong disorder phase exists (see \cite{cf:HHP}).
\end{rem}

\section{Existence of the quenched connective constant and monotonicity properties}

\subsection{Proof of Proposition \ref{taex}}
We prove in this section that $\mu_d(p)$ is well-defined. 
Some intermediate lemmata are proved for general infinite connected graphs.
\begin{lemma}\label{conecg}
For an infinite connected graph $\mathcal C$ with bounded degree and $x$ a vertex of $\mathcal C$, we define
$$Z_N(x):=| \{ \text{ self-avoiding path of length $N$ and starting from $x$ on $\mathcal C$ }\}|.$$
Then the quantity
\begin{equation}
 \mu(\mathcal C):= \limsup_{N\to\infty} ( Z_N(x))^{1/N}
\end{equation}
is a constant function of $x$.
\end{lemma}

Before starting the proof, let us mention that the formula \eqref{hme} below, which is the main step of the proof,  also appears in a paper of  Hammersley  \cite[Equation (13)]{cf:Ham}, where it is proved using the same argument.

\begin{proof}
As $\mathcal C$ is by definition connected, it
is sufficient to show that $\limsup_{N\to\infty} ( Z_N(x))^{1/N}$ takes the same value for every pair of neighbors.
Let $x$ and $x'$ be connected by an edge of $\mathcal C$.
Let $\bar Z_N(x)$ be the number of 
self-avoiding paths of length $N$ starting from $x$ and never visiting $x'$.
Let $Y_N(x,x')$ be the number of self-avoiding paths of length $N$ starting from $x$ and ending at $x'$.
One has trivially  (noting that $Y_N(x,x')=Y_N(x',x)$)
\begin{equation}\begin{split}
 Y_N(x,x') &\le Z_N(x'),\\
 \bar Z_{N}(x)&\le  Z_{N+1}(x').
\end{split}
\end{equation}
The second inequality simply says that $\bar Z_N(x)$ also counts the number of paths
 of length $N+1$ starting from $x'$
whose first step is $x$.

Let $k$ denote the time of the visit of $S$ to $x'$. Decomposing $Z_N(x)$ according to possible values $k$ one gets
\begin{equation}\label{hme}
 Z_N(x)\le \sum_{k=1}^N Y_k(x,x') Z_{N-k}(x')+ \bar Z_N(x)
\le \sum_{k=1}^N  Z_{k}(x')Z_{N-k}(x')+ Z_{N+1}(x'),
\end{equation}
which is enough to conclude that 
\begin{equation}
 \limsup_{N\to \infty}  Z_N(x)^{1/N}\le  \limsup_{N\to \infty}  Z_N(x')^{1/N},
\end{equation}
and thus by symmetry, that the two are equal.

\end{proof}

\begin{rem}\rm
Although to our knowledge,  Lemma \ref{conecG} has not appeared before literature, tin
\end{rem}

What remains to be done is to prove that when $\mathcal C$  is a supercritical percolation cluster,
$\mu(\mathcal C)$ is a.s.\ non-random.
 A first step is to show that $\mu(\mathcal C)$
is non-sensitive to individual edge addition (and thus to edge removal).

\begin{lemma}\label{edgead}
Let  $\mathcal C$ be an infinite connected graph with bounded degree and $x,x'\in \mathcal C$ that are not linked by an edge.
Then define

\begin{itemize}
 \item [(i)] $\mathcal C'$ to be the graph constructed from $\mathcal C$ by adding a new vertex called $y$ and
and an edge $(x,y)$ linking $x$ to $y$.
 \item [(ii)] $\mathcal C''$ to be the graph with same set of vertices as $\mathcal C$, and an added edge: $(x,x')$.

\end{itemize}
We have
\begin{equation}
 \mu(\mathcal C)= \mu(\mathcal C')=\mu(\mathcal C'').
\end{equation}
\end{lemma}

\begin{proof}

We call $Z'_N$ and $Z''_N$ the number of self-avoiding path of length $N$ on $\mathcal C'$ resp.\ $\mathcal C''$.
Note that for $N\ge 2$, we have $Z_N(x)=Z'_N(x)$ and thus $\mu(\mathcal C)= \mu(\mathcal C')$.

\medskip

Now consider the case of $\mathcal C''$.
Decomposing over paths that use the edge $(x,x')$ and those that don't, one gets

\begin{equation}
 Z_N(x)\le Z''_N(x)\le Z_{N-1}(x')+Z_N(x).
\end{equation}
Taking the above inequality to the power $\frac{1}{N}$ and passing to the $\limsup$, one gets that
\begin{equation}
\limsup_{N\to \infty}(Z''_N(x))^{1/N}=\mu(\mathcal C),
\end{equation}
which ends the proof.
\end{proof}

We are now ready to conclude the proof of Proposition \ref{taex}. In what follows $\mathcal C$ denotes again the infinite connected component of the percolation process.
By uniqueness of the infinite cluster,
modifying the environment on any finite set of edges only adds or deletes finitely many edges to $\mathcal C$, so that the 
new cluster can be obtained from the old one by performing Operations $(i)$ or $(ii)$ of Lemma \ref{edgead} or their converse a finite number of times.
Hence by Lemma \ref{edgead}, $\mu(\mathcal C(\go))$ is measurable with respect to the tail sigma-algebra of the field $(\go_e)_{e\in E_d}$, 
which is known to be trivial. Hence it is non-random.

\qed

\subsection{Proof of monotonocity of $(\mu_d(p)/p\mu_d(1))$}\label{tootoo}
To prove the monotonicity of the ratio between the quenched and annealed connectivity constants, we use a coupling argument that is quite standard.
We couple the two measures $\bbP_p$ and $\bbP_p'$ for $p_c<p<p'<1$ as follows: let $E_d$ denote the set of edges of $\bbZ^d$, we consider a field $(X(e))_{e\in E_d}$ of IID  random variables (call $\bbE$ the law of the field) 
that are
uniformly distributed in $[0,1]$.
Then one sets
\begin{equation}
 \go_p(e)=\ind_{\{X(e)\le p\}}, \quad \go_{p'}(e)=\ind_{\{X(e)\le p'\}}.
\end{equation}
With this construction, the infinite open clusters of $\go_p$ and $\go_p'$ satisfy   $\mathcal C_p\subset \mathcal C_{p'}$.
Moreover, if one sets
\begin{equation}
 \mathcal F_{p'}=\sigma(\go_{p'}(e), e\in \bbZ^d),
\end{equation}
then one has
\begin{equation}
 \bbE \left[ Z_N(\go_p)\ | \ \mathcal F_{p'}\right]=\sum_{S \in \mathcal S_N} \bbP\left[\ind_{S\text { is open for $\go_p$}} \ | \ \mathcal F_{p'}\right].
\end{equation}
The reader can then check that
\begin{equation}
 \bbP\left[\ind_{S\text { is open for $\go_p$} } \ | \mathcal F_{p'}\right]= \ind_{S\text { is open for $\go_{p'}$}}(p/p')^N.
\end{equation}
Summing over $S\in \mathcal S_N$ gives
\begin{equation}
  \bbE \left[ Z_N(\go_p)\ | \ \mathcal F_{p'}\right]= (p/p')^N Z_N(\go_p').
\end{equation}
Using the Borel-Cantelli Lemma, one gets that for all $N$ large enough

\begin{equation}
 Z_N(\go_p)\le N^2 (p/p')^N Z_N(\go_p'),
\end{equation}
which implies
 \begin{equation}
  \frac{\mu_d(p)}{p}\le \frac{\mu_d(p')}{p'}.
 \end{equation}
.

\qed

\section{Proof of the main result: Theorem \ref{superteo}} \label{pprroof}

In this section we focus on the proof of the non-equality between the quenched and annealed connective constants, or Equation \eqref{trucpricn}.
For the proof of the strict monotonicity of $\mu_d(p)/p\mu_d$ we refer to Section \ref{newsec}.

\subsection{About the proof}

The main ingredients of the proof are fractional moment, coarse-graining and change of measure. This combination of  ingredients 
has been used several times in 
the recent past in the study of disordered systems with the aim of comparing quenched and annealed behaviors. 
The method was first introduced in \cite{cf:GLT} for the study of 
disordered pinning on a hierarchical lattice. It was then improved in \cite{cf:Tcg} (introduction of an efficient coarse-graining
on a non-hierarchical setup) and in \cite{cf:GLT2, cf:GLT3} (improvement of the change of measure argument by introducing a multibody interaction). 
It  has also
been successfully adapted to a variety of models and we can cite a few contributions on the random walk pinning model \cite{cf:BT, cf:BS}, 
directed polymers
in a random environment
\cite{cf:L}, stretched polymers \cite{cf:Z}, random walk in a random environment \cite{cf:YZ} (for technical details \cite[Section 4]{cf:L}   is probably the most related to what 
we are doing here).
\medskip

The major difference between all the models mentioned above  and the one we are studying here is the amount of knowledge that one 
has on the annealed model. The annealed version of all of the models is either the directed simple random walk 
or a mildly modified version of it (e.g.\ in \cite{cf:YZ, cf:Z}) and
the proof uses the rather precise knowledge that one has about the
simple random walk (e.g.\ the central limit theorem) to draw conclusions. In \cite{cf:GLT2, cf:GLT3, cf:L, cf:YZ, cf:Z}, in the $(1+2)$- or $3$-dimensional case,
the need for a more refined change of measure is due to the fact that we are at the critical dimension, where extra precision is needed. 

\medskip

On the contrary, here, even though we are not at the critical dimension (recall that we believe that the result also holds in dimension $3$), 
similar refinements have to be used for a different reason.
The problem is rendered more difficult by the fact that almost nothing has been rigorously proved for the planar self-avoiding walk in spite of numerous conjectures
(e.g.\ we don't have a good control on $\bbE\left[ Z_N \right]$ beyond the exponential scale, and we almost have
no rigorous knowledge about the trajectory properties). For this reason, we  need a method that covers all of the worst-case scenarios.
As a consequence we believe that the quantitative estimate that we derive from our method is almost irrelevant.
The techniques we use rely neither on the peculiar features of percolation nor on the lattice and thus are quite easy to export to other $2$-dimensional models
(see Sections \ref{omega} and \ref{export}).

\medskip

The main novelties in the proof are:
\begin{itemize}
 \item A new type of coarse-graining, that allows to take into account the fact that contrary to the directed walk case, the walk can and will return to regions that it has already visited  (in \cite{cf:YZ, cf:Z} 
even if the walks are not directed
the situation is different because they naturally stretch along one direction). 
 \item A new type of change of measure (inspired from the one used in \cite{cf:L}, 
 but modified to adapt our new setup) and a new method 
to estimate the
gain given with this change of measure. 
\end{itemize}

The rest of the proof is organized as follows: In Section \ref{fredy}, in order to familiarize the reader with our change of measure technique, we 
prove a simple result (Proposition \ref{fredo}) which gives a relation between the volume exponent and the behavior of the partition function at small dilution. 
In Section \ref{cganim}, we explain what we mean by fractional moments, and introduce 
our coarse-grained decomposition.
It associates a coarse-grained lattice animal with each trajectory. This reduces the proof of \eqref{trucpricn} to Proposition \ref{fundamental},
which controls the contribution of each animal.
In Section \ref{measur} we give the main idea of the proof of Proposition \ref{fundamental}, which is proved in 
Sections \ref{smallm} and \ref{bigm} for small and large values of $m$ respectively, $m$ being the size of the coarse-grained animal.

\subsection{Change of measure without coarse graining: Heuristics and conjecture support}\label{fredy}

In order to give the reader a clear view of the ideas hiding behind our proof strategy, we want to prove first a simpler result that partially confirms the prediction of \cite{cf:LDM}.
Note also that it suggests that the use of Harris Criterion in  \cite{cf:CK} is valid, as our result establishes a relation between the relevance of disorder and the positivity of the specific heat exponent $2-d\nu_d$.

Set 
\begin{equation}
\mathcal S_N(\alpha):= \{ S\in \mathcal S_N \ | \ \max_{n\in[0,N]} \|S_n\| \le N^\alpha \},
\end{equation}
where $\|S_n\|$ is the $l_{\infty}$-norm of $S_n$, and $s_N(\alpha):=|\mathcal S_N(\alpha)|$. We define the volume exponent

\begin{equation}
\nu_d:=\inf \{ \alpha \  | \  \liminf_{N\to \infty} s_N(\alpha)/s_N=1\}.
\end{equation}

\begin{proposition}\label{fredo}

Assume that the relation $d\nu_d<2$ is satisfied. Then for all $p<1$, in probability,
\begin{equation}\label{lacedroi}
\frac{Z_N}{\bbE_p[Z_N(\go)]}\to 0.
\end{equation}

\medskip

In particular when $d=3$, if $\nu_3<2/3$ then \eqref{lacedroi} holds for all $p<1$.

\end{proposition}

\begin{proof}
Choose $\alpha>\nu_d$ such that $d\alpha<2$ and set 

\begin{equation}
Z_N=\sum_{S\in \mathcal S_N(\alpha)}\ind_{S \text{ is open}}+\sum_{S\in \mathcal S_N\setminus \mathcal S_N(\alpha)}\ind_{S \text{ is open}}=:Z^{(1)}_N+
Z_N^{(2)}.
\end{equation}
By the definition of $\nu$ one has 
$$\bbE[Z_N^{(2)}]/\bbE[Z_N]=1-(s_N(\alpha)/s_N)\stackrel{N\to\infty}{\longrightarrow}0,$$
and hence $Z_N^{(2)}/\bbE[Z_N]$ tends to zero in probability.

\medskip

To show that $Z_N^{(1)}/\bbE[Z_N]$ also tends to zero in probability, we prove that

\begin{equation}\label{sratf}
\lim_{N\to \infty}\bbE\left[\sqrt{Z_N^{(1)}}\right]/\sqrt{\bbE[Z_N]}=0.
\end{equation}

We introduce a new measure $\tilde \bbP_N$ on the environment
which modifies the law of $\go$ inside the box $[-N^{\alpha},N^\alpha]^d$.
Under $\tilde \bbP_N$, the $\go(e)$ are still independent  Bernoulli variables but they are not identically distributed,
and the probability of being open is lower for edges in $[-N^{\alpha},N^\alpha]^d$, or more precisely
\begin{equation}
\tilde \bbP_N(\go(e)=1):=p'\ind_{e\in[-N^{\alpha},N^\alpha]^d}+p\ind_{e\notin[-N^{\alpha},N^\alpha]^d}
\end{equation}
where $e\in[-N^{\alpha},N^\alpha]^d$ means that both ends of $e$ are in $[-N^{\alpha},N^\alpha]^d$, and 
$$p'=p'(N,p):=p(1-N^{-d\alpha/2}).$$ 

The Radon--Nikodym derivative of $\tilde\bbP_N$ with respect to $\bbP$ is equal to
\begin{equation}\label{denschag}
\frac{\dd \tilde \bbP_N}{\dd \bbP}(\go):=\left(\frac{p'(1-p)}{(1-p')p}\right)^{\sum_{e\in[-N^{\alpha},N^\alpha]^d}\go(e)}
\left(\frac{1-p'}{1-p}\right)^{\#\{e\in[-N^{\alpha},N^\alpha]^d\}}.
\end{equation}
With this choice of $p'$, the probability law $\tilde \bbP_N$ is not too different from $\bbP$ (the total variation distance between the two is bounded away from one when $N$ tends to infinity), but $\tilde \bbE_N\left[Z_N^{(1)}\right]$ is much smaller than $\bbE[Z_N^{(1)}]$, and this is what is crucial to make our proof work.

\medskip

Using the Cauchy-Schwartz inequality 

\begin{equation}\label{CSSS}
\bbE\left[\sqrt{Z_N^{(1)}}\right]=\tilde \bbE_N\left[\left(Z_N^{(1)}\right)^{1/2}
\frac{\dd  \bbP}{\dd \tilde \bbP_N}(\go)\right]\le \left(\tilde \bbE_N\left[Z_N^{(1)}\right]\right)^{1/2}\left(\bbE\left[\frac{\dd \bbP}{\dd \tilde  \bbP_N}(\go)\right]\right)^{1/2}.
\end{equation}

For a path in $S\in \mathcal S_N(\alpha)$ we have $\tilde \bbP_N(S \text{ is open })=(p')^N$, and hence

\begin{equation}\label{abcd}
\tilde \bbE_N\left[Z_N^{(1)}\right]=(p')^N s_N(\alpha)\le (1-N^{-d\alpha/2})^N \bbE[Z_N]\le \exp(-N^{1-(d\alpha)/2})\bbE[Z_N].
\end{equation}
Moreover 
\begin{multline}\label{fghi}
\bbE\left[\frac{\dd \bbP}{\dd \tilde  \bbP_N}(\go)\right]=\left(\frac{p^2(1-p')+p'(1-p)^2}{p'(1-p')} \right)^{\#\{e\in[-N^{\alpha},N^\alpha]^d\}}
\\=\left[p(1-N^{-d\alpha/2})^{-1}+(1-p)\left(1+\frac{pN^{-\alpha d/2}}{1-p}\right)^{-1}\right]^{\#\{e\in[-N^{\alpha},N^\alpha]^d\}}
\\ 
\le \left[1+2\left(p+\frac{p^2}{1-p}\right)N^{-d\alpha}\right]^{\#\{e\in[-N^{\alpha},N^\alpha]^d\}}\\
\le \exp\left( d 2^{d+2}\left(p+\frac{p^2}{1-p}\right)\right).
\end{multline}
The first inequality above uses second order Taylor expansions of $(1\pm x)^{-1}$ and thus is valid for fixed $p$, when $N$ is sufficiently large.
The last inequality uses the fact that 
$$\#\{e\in[-N^{\alpha},N^\alpha]^d\}\le d 2^{d+1} N^{\alpha}.$$
Hence combining \eqref{CSSS} with \eqref{abcd} and \eqref{fghi} we obtain
\begin{equation}
\frac{ \bbE\left[\sqrt{Z_N^{(1)}}\right]}{\sqrt{\bbE[Z_N]}}\le C(d,p)\exp(-N^{1-(d\alpha)/2}/2).
\end{equation}
for some constant $C(d,p)$. As $1-(d\alpha)/2>0$ this shows that \eqref{sratf} holds and this concludes the proof.
\end{proof}

A drawback of the result presented in this section is that it is not even known rigorously that $\nu<1$ for $d=2$ so even in this case it cannot apply.
 Also -- and this is the most important point --  it does not give the exponential decay of $Z_N/ \bbE[Z_N]$, but only its convergence to zero in probability.

\medskip

To prove the exponential decay of $Z_N / \bbE[Z_N]$ we use the change of measure technique of the proof of the proposition above, but combine it with a coarse graining argument:
we divide the lattice in big cells of width $N_0$ and then apply a change of measure similar to the one in\eqref{denschag} to each cell. This technique will allow us to gain better information on the decay of $Z_N / \bbE_p[Z_N(\go)$. Changing the density of the open edges as in \eqref{denschag} is however not always sufficient, and we will have
to use a more subtle change of measure that induces negative correlation between the edges in one cell (Section \ref{bigm}).

\subsection{The fractional moment method and animal decomposition}\label{cganim}

Fractional moment is a technique extensively used by physicists that consists in estimating 
non-integer moments of a partition function in order to get information about it.
From now on we omit implicit dependences on $p$ in our notation when it does not affect understanding.

\medskip

In our case, the fractional moment method consists in saying that to prove our result \eqref{trucpricn}, it is sufficient to prove that
there exists $\theta\in (0,1)$ and $b<1$ such that, for $N$ large enough

\begin{equation}\label{trucfacil}
 \bbE\left[(Z_N)^{\theta}\right]\le b^{N\theta} \bbE\left[Z_N\right]^{\theta}=\left[s_N(bp)^N\right]^\theta.
\end{equation}
Indeed by the Borel--Cantelli Lemma (combined with the Markov inequality), \eqref{trucfacil} implies that a.s.\ for sufficiently large $N$
\begin{equation}\label{trocat}
 Z_N\le N^{2/\theta}(bp)^N s_N,
\end{equation}
so that passing to the $\limsup$ one gets 
\begin{equation}\label{tras}
\mu_2(p)\le bp\mu_2(1).
\end{equation}

\medskip

We consider the following coarse graining procedure which associates  a lattice-animal on a rescaled lattice with each path. 
Set 

\begin{equation}\label{nodef}
N_0:= \exp\left(\frac{C_2}{(1-p)^2}\right),
\end{equation}
where $C_2$ is a constant (independent of $p$) whose value will be fixed at the end of the proof,
and let us partition the set of edges $E_d$ into squares of side length $N_0$.
More precisely, let $r(e)$ denote the smaller end (for the lexicographical order on $\bbZ^2$) 
of an edge $e\in E_2$ , and for $x\in \bbZ^2$ 
set
\begin{equation}
I_x:=\{e\in E_2 \ | \ r(e)\in \left(N_0 x+[0,N_0)^2\right)\}.
\end{equation}
Now, one associates with each path $S$ the set of squares $I_x$ that it visits.
Set 
\begin{equation}
 A(S):=\{x\in \bbZ^2 \ | \ \exists n \in [0,N-1], (S_n,S_{n+1})\in I_x\}. 
\end{equation}

Note that $A(S)$ is a connected subset of $\bbZ^2$ that contains the origin (sometimes called a site-animal), and that 
\begin{equation}
\lceil N/N^2_0 \rceil \le |A(S)| \le 9\lceil N/N_0 \rceil.
\end{equation}
The upper bound comes from the fact that in $N_0$ steps, one cannot visit more than $9$ different $I_x$'s
(the one from which one starts and of its neighbors), whereas the lower bound simply uses the fact that there are only 
$N_0^2$ sites to visit in each square.
From now on, one drops the integer parts in the notation for simplicity.
Let $\mathfrak{A}_m$ be the set of  connected subsets of $\bbZ^2$ of size $m$ containing the origin, and $a_m:=|\mathfrak{A}_m|$.
For each animal $\mathcal A$ we set 
\begin{equation}
\mathcal S_N(\mathcal A):=\{S\in \mathcal S_N \ | \ A(S)=\mathcal{A}\}.
\end{equation}
Then one decomposes the partition function according to the contribution of each animal:
\begin{equation}\label{crcr}
 Z_N=\sum_{m=N/N^2_0}^{9((N/N_0)+1)}\sum_{\mathcal A \in \mathfrak{A}_m}\sum_{S\in \mathcal S_N(\mathcal A)} 
\ind\{\text{$S$ is open}\}=:\sum_{m=N/N^2_0}^{9((N/N_0)+1)}\sum_{\mathcal A \in \mathfrak{A}_m} Z_N(\mathcal A).
\end{equation}
We use the following trick: for any $\theta<1$ and any summable sequence of positive numbers $(a_n)_{n\in \bbN}$ one has
\begin{equation}
 (\sum_{n\in\bbN} a_n)^{\theta}\le \sum_{n\in \bbN} a_n^\theta.
\end{equation}
Thus applying this to \eqref{crcr} and averaging one gets

\begin{equation}
 \bbE\left[Z_N^{\theta}\right]\le \sum_{m=N/N^2_0}^{9((N/N_0)+1)}\sum_{\mathcal A \in \mathfrak{A}_m}
\bbE\left[Z_N(\mathcal A)^{\theta}\right].
\end{equation}

There are at most exponentially many animals of size $m$.
Here, we use the crude estimate $a_m\le 49^m$ (see e.g.\ \cite[(2.4) p. 81]{cf:G} for a proof: the definition
of lattice animal given there differs, but the bound still applies).
Hence 

\begin{equation}\label{efgrd}
  \bbE\left[Z_N^{\theta}\right]\le N \max_{m\in[N/N^2_0,9((N/N_0)+1)]}49^m\max_{\mathcal A \in \mathfrak{A}_m}\bbE\left[Z_N(\mathcal A)^{\theta}\right].
\end{equation}

Thus, in order to prove \eqref{trucfacil} it is 
sufficient to prove that \ $\bbE\left[Z_N(\mathcal A)^{\theta}\right]\le 100^{-m}(p^N s_N)^{\theta}$ for every $m$ and $\mathcal A$.
This is the key part of the proof.

\begin{proposition}\label{fundamental}
 For $\theta=1/2$ and  for every $\mathcal A\in \mathfrak{A}_m$ 
\begin{equation}\label{zactos}
 \bbE\left[Z_N(\mathcal A)^{\theta}\right]\le p^{N\theta} s^{\theta}_N 100^{-m},
\end{equation}
if the constant $C_2$ is chosen large enough.
\end{proposition}

The above proposition combined with equation \eqref{efgrd} implies that for $\theta=1/2$ 
\begin{equation}
 \bbE\left[Z_N^{\theta}\right]\le N p^{N\theta} s^{\theta}_N  2^{-N/N^2_0},
\end{equation}
which implies \eqref{trucfacil} for $b=5^{-1/N^2_0}$ ($5$ is chosen instead of $4$ to absorb the extra  $N$ factor appearing)  and $N$ large enough.
Thus from \eqref{trucfacil}, \eqref{trocat} and \eqref{tras} we get
\begin{equation}
\mu_2(p)=\limsup_{N\to \infty} (Z_N)^{1/N}\le 4^{-1/N_0^2} p\mu_2(1)<p\mu_2(1).
\end{equation}
Hence there exists a constant $c$ such that for all $p$
\begin{equation}\label{irrelvt}
 p\mu_2(1)-\mu_2(p)\ge \frac{c}{N^2_0}= c\exp\left(-\frac{2C_2}{(1-p)^2}\right).
\end{equation}

\subsection{Change of measure strategies}\label{measur}

Let us explain our strategy behing the proof of Proposition \ref{fundamental}.
It is based on a change a measure argument. The fundamental idea is that if $\bbE[\sqrt{Z_N(\cA)}]$ is much smaller than 
$\sqrt{\bbE[Z_N(\cA)]}$, it must be because there is a small set of $\go$ (of small $\bbP$ probability)
giving  the main contribution to $\bbE[Z_N(\cA)]$.
We want to introduce a random variable $f_{\cA}(\go)$ which takes a low value for these untypical environment.

\begin{lemma}\label{coaer}
For any $\mathcal A$ and any positive random variable $f_{\mathcal A}$ one has
\begin{equation}
 \bbE\left[Z_N(\mathcal A)^{1/2}\right]
\le \bbE\left[f_{\mathcal A}Z_N(\mathcal A)\right]^{1/2}
\left(\bbE\left[(f_{\mathcal A})^{-1}\right]\right)^{1/2}.
\end{equation}
\end{lemma}
\begin{proof}
Apply the Cauchy--Schwartz inequality to the product $(Z_N(\mathcal A)f_{\mathcal A})^{1/2}\times f_{\mathcal A}^{-1/2}$.
\end{proof}

Note that if $f_{\mathcal A}$ has finite expectation, then it can be thought of as a probability density after renormalization, so 
this operation can indeed be interpreted  as a change of measure.
The main problem then is to find an efficient change of measure for which the cost of the change 
$\bbE\left[(f_{\mathcal A})^{-1}\right]$ is much smaller than the benefit 
one gets on $\bbE\left[f_{\mathcal A}Z_N(\mathcal A)\right]$.  

\medskip

In order to get an exponential decay in $m$ for $\bbE\left[Z_N(\mathcal A)^{1/2}\right]/\bbE\left[Z_N(\mathcal A)\right]^{1/2}$,
the good choice is to take $f_{\mathcal A}$ to be  a product over all $x\in \mathcal A$ of functions of the environment of each block $(\go_e)_{e\in I_x}$.

\medskip

One possibility is to diminish the intensity of open edges in $\cup_{x\in \mathcal A} I_x$, like in Section \ref{fredy}, simply by choosing
\begin{equation}
 f_{\mathcal A}(\go)\asymp \gl^{\# \text{ open edges }}
\end{equation}
for some $\gl<1$.
This turns out to be a good choice 
when the animal $\mathcal A$
considered is relatively small, but it does not give a good result when $m=|\mathcal A|$ is of order $N/N_0$, even after optimizing the value of $\gl$.
A more efficient strategy in that case is to induce negative correlation between the $\go(e)$ that decays with the distance instead of reducing 
the intensity of edge opening.

\medskip

This idea was first used in \cite{cf:GLT2}. There and in all related works, 
the induced negative correlations were chosen to be proportional 
to the Green function
of the underlying process (either a renewal process in \cite{cf:GLT2} or a directed random walk in \cite{cf:L,cf:YZ,cf:Z}). 
Here, the situation is a bit different
as one does not have any information on the underlying process:
 therefore the choice of correlation (i.e.\ of the coefficients  in the quadratic form
 $Q$ in equation \eqref{qudrtc}) is done via an optimization procedure,
 so that it lowers significantly the probability of $\bbP[S \text{ is open}]$ 
for every possible path (and not only the more probable ones).
\medskip

We adopt the first strategy when $m\le N/[N_0(\log N_0)^{1/4}]$, and the second one when $m> [N_0(\log N_0)^{1/4}]$.

\subsection{Proof of Proposition \ref{fundamental} for small values of $m$} \label{smallm}

In this section we assume that 
\begin{equation}\label{zerpis}
 m\le N/[N_0(\log N_0)^{1/4}].
\end{equation}
We choose to modify the environment in $$I_{\mathcal A}:=\bigcup_{x\in \mathcal A} I_x,$$
by augmenting the intensity of the edge dilution.
We choose the probability of an edge being open under the new measure to be equal to
$$p':=\frac{\gl p}{1-p(1-\gl)},$$ where $\gl<1$ is chosen such that
\begin{equation}\label{zerpi}
(1-\gl)\sqrt{1-p}N_0=1.
\end{equation}
As there are $2N^2_0$ edges in each block $I_x$, the density function corresponding to this change of measure is given by

\begin{equation}\label{smallf}
 f_{\mathcal A}:=\frac{\gl^{\#\{ \text{ open edges in } I_{\mathcal A}\}}}{[1-p(1-\gl)]^{2mN^2_0}}.
\end{equation}
Then, we can estimate the cost of the change of measure
\begin{multline}\label{stimac}
 \bbE[f_{\mathcal A}^{-1}]= \left[\left(1+\frac{p}{\gl}(1-\gl)\right)(1-p(1-\gl))\right]^{2mN^2_0}\\
=\left(1+\frac{p}{\gl}(1-p)(1-\gl)^2\right)^{2mN^2_0}\le \exp\left(\frac{2p}{\gl}m\right)\le 10^m,
\end{multline}
where for the first inequality we use $\log(1+x)\le x$ and the identity \eqref{zerpi}, and for the second one, the 
fact that $\gl$ is close to one (if $C_2$ is chosen large enough).

\medskip

The probability that a path of length $N$ 
whose edges are all in $I_{\mathcal A}$ is open under the modified measure is 
equal to $[\gl p/(1-p(1-\gl)]^N$, so 

\begin{equation}\label{roselmc}
 \bbE[f_{\mathcal A}Z_N(\mathcal A)]=|\mathcal S_N(\mathcal A)|p^N\left(\frac{\gl}{1-p(1-\gl)}\right)^N.
\end{equation}
Then we remark that $|\mathcal S_N(\mathcal A)|\le s_N$ and that 
\begin{multline}
 \left(\frac{\gl}{1-p(1-\gl)}\right)^N=\left(1-\frac{(1-\gl)(1-p)}{1-p(1-\gl)}\right)^N
 \le \exp\left(-N(1-\gl)(1-p)\right)\\
  \stackrel{\eqref{zerpis}}{\le}
 \exp\left(-m N_0 (\log N_0)^{1/4}(1-\gl)(1-p)\right)\\ \stackrel{\eqref{zerpi}}{=}\exp\left(-m(\log N_0)^{1/4}\sqrt{1-p}\right)\stackrel{\eqref{nodef}}{=}\exp(-m C^{1/4}_2).
\end{multline}
Hence, if $C_2$ is large enough,
\begin{equation}
 \bbE[f_{\mathcal A}Z_N(\mathcal A)]\le s_N p^N \exp(-m C^{1/4}_2)\le s_N p^N 10^{-5m}.
\end{equation}
Combining this with \eqref{stimac} and Lemma \ref{coaer}, we get \eqref{zactos}.
\qed

\medskip

\begin{rem}\rm
Let us now comment on why we believe that \eqref{irrelvt} is suboptimal.
The idea is that in fact, if the typical scaling of a self-avoiding path is $N^{\nu}$,
then (at heuristic level) the typical number $m$ of blocks visited  is of order $N/(N_0^{1/\nu})$,
and the argument above works for a much lower value of $N_0$ (of order $(1-p)^{\frac{2\nu}{2(1-\nu)}}$), giving then a much better upper bound for all
$\mu_2(p)$.
Bringing this kind of argument to a rigorous ground would require 
very detailed knowledge of the behavior of the self-avoiding 
walk.

\end{rem}

\subsection{Proof of Proposition \ref{fundamental} for large values of $m$} \label{bigm}

Even when trying to optimize over the value of $\gl$ or when taking a much larger value for $N_0$
the preceding method fails when the size of the animals is of order 
$N/N_0$, and we have to apply a different method in this case. 
Throughout this section we will assume that 
\begin{equation}\label{zerpos}
 m> N/[N_0(\log N_0)^{1/4}].
\end{equation}

Our proof is still based on Lemma \ref{coaer}, but
the construction of our $f_{\mathcal A}$ is a bit more complicated in this case, 
and for notational convenience
we do not normalize it: it is not a probability density.
\medskip

First, given an animal $\mathcal A$ of size $m$, one can extract a set of vertices $\bar {\mathcal A}$ 
of size $m/13$ such that the vertices of
$\mathcal A$ have disjoint $l_{\infty}$ neighborhood, i.e.\ such that
\begin{equation}
 \forall x,y \in \bar {\mathcal A},\quad  x\ne y, \text{ and } |x-y|_{\infty} \ge 3,
\end{equation}
where $|x|_{\infty}=\max(|x_1|,|x_2|)$.

\medskip

For instance, one can construct $\bar {\mathcal A}$ by picking vertices in $\mathcal A$ iteratively as follows:
at each step we pick the smallest available vertex according to the lexicographical order in $\bbZ^2$, 
and make all the vertices 
at a $l_\infty$ distance  $2$ or less of this vertex unavailable for future picks. 
As at most $13$ vertices are made unavailable at each step, we can keep 
this procedure going 
for $m/13$ steps to get $\bar{\mathcal A}$.
For $x\in \bbZ^2$ set 
\begin{equation}
 \bar I_x:=I_x\cup\left(\bigcup_{\{ y\ | \ |y-x|_{\infty}=1\}} I_x\right). 
\end{equation}
We define the distance $d$ between edges to be the Euclidean distance between their midpoints.
Given $K$ a (large) constant, we define $f_x$ to be a function of $\go$, that depends only on $\go|_{\bar I_x}$:
first set $Q_x$ to be the following random quadratic form 
\begin{equation}\label{qudrtc}
Q_x(\go):=\frac{1}{(1-p) N_0\sqrt{\log N_0}}\sumtwo{e,e'\in  \bar I_x}{e'\ne e}\frac{1}{d(e,e')}(\go(e)-p)(\go(e')-p),
\end{equation}
and then define
\begin{equation}\label{fxxx}
 f_x(\go):=\exp\left(-K\ind_{\{Q_x(\go)\ge e^{K^2}\}}\right).
\end{equation}
Finally,  set 
\begin{equation}\label{faaa}
f_{\mathcal A}(\go):=\prod_{x\in \bar {\mathcal A}}  f_x(\go).
\end{equation}
In order to use Lemma \ref{coaer}, one needs to bound  $\bbE\left[(f_{\mathcal A})^{-1}\right]$ from above.

\begin{lemma}\label{coaer2}
If $K$ is chosen sufficiently large, then for every $A\in \mathfrak{A}_m$, we have
\begin{equation}
   \bbE\left[(f_{\mathcal A})^{-1}\right]\le 2^{m/13}.
\end{equation}
\end{lemma}
\begin{proof}

The function $f_{\mathcal A}(\go)$ is a product of $m/13$ IID random variables ($f_x(\go)$, $x \in \bar{\mathcal A}$)
(since our choice for $\bar{\mathcal A}$, the blocks $(\bar I_x)_{x\in \bar{\mathcal A}}$
are disjoint).
Thus

\begin{equation}
 \bbE\left[(f_{\mathcal A})^{-1}\right]=\bbE\left[(f_{0})^{-1}\right]^{m/13}.
\end{equation}

It remains to prove that $\bbE\left[(f_{0})^{-1}\right]\le 2$, and for this purpose
it is sufficient to estimate the variance of $Q_0(\go)$.
First note that $\bbE[Q_0(\go)]=0$, and that only the diagonal terms of the double 
sum that appears when expanding $Q^2_0(\go)$ contribute to the second 
moment.
Note also that the maximal distance between two edges in $\bar I_0$ is less than $5 N_0$, so

\begin{multline}\label{zrz}
 \bbE[Q_0(\go)^2]=\frac{1}{(1-p)^2 N_0^2\log N_0}\sumtwo{e,e'\in  \bar I_0}{e'\ne e}
\frac{1}{d(e,e')^2}p^2(1-p)^2\\
\le 
\frac{1}{N^2_0\log N_0}\sum_{e\in \bar I_0}\sum_{\{e'\ne e | d(e,e')\le 5 N_0 \}}\frac{1}{d(e,e')^2}
\le C_1,
\end{multline}
 where $C_1$ is a universal constant that is independent of $p$ and $N_0$.
Thus, by the Chebytcheff inequality, if $K$ is large enough (independently of all parameters of the problem)
\begin{equation}
  \bbE\left[\left(f_{0}(\go)\right)^{-1}\right]= 1+(e^{K}-1)\bbP\left[Q_0(\go)\ge e^{K^2}\right]
\le 1+C_1(e^K-1)e^{-2K^2}\le 2.
\end{equation}

\end{proof}

It remains to estimate $\bbE\left[f_{\mathcal A}Z_N(\mathcal A)\right]$, which is the more delicate part.
We do so by bounding uniformly the contribution of each path.

\begin{lemma}\label{crdar}
For any $S\in \mathcal S_N(\mathcal A)$, we have
\begin{equation}\label{crdareq}
 \bbE\left[f_{\mathcal A}(\go)\ind_{\{ S \text{ is open }\}}\right]
=p^N  \bbE\left[f_{\mathcal A}(\go)\ | \ S \text{ is open }\right]\le 
p^ N 20000^{-m}.
\end{equation}
\end{lemma}
Lemma \ref{crdar}, combined with the trivial bound $|\mathcal S_N(\mathcal A)|\le s_N$ gives
\begin{equation}
  \bbE\left[f_{\mathcal A}(\go)Z_N(\mathcal A)\right]\le p^N s_N 20000^{-m},
\end{equation}
so that together with Lemmata \ref{coaer} and \ref{coaer2} one obtains
\begin{equation}
   \bbE\left[Z_N(\mathcal A)^{1/2}\right]\le p^{N/2} s_N^{N/2} 100^{-m},
\end{equation}
which proves Proposition \ref{fundamental}.

\begin{proof}[Proof of Lemma \ref{crdar}]
Note that even after conditioning on $S$ being open, $f_{\mathcal A}(\go)$ is still a product of 
independent variables (though the $f_x(\go)$ are not identically distributed any more), so that
\begin{equation}\label{trtrsi}
\bbE\left[f_{\mathcal A}(\go)\ | \ S \text{ is open } \right]=
\prod_{x\in \bar {\mathcal A}} \bbE\left[ f_x(\go)\ |  \ S \text{ is open }\right].
\end{equation}

Our idea is then to show that most of the terms in the product $\bbE\left[ f_x(\go)\ |  \ S \text{ is open }\right]$ are small.
We do so by showing that conditioning on the event $\{S \text{ is open }\}$ makes the expectation of
$Q_x$ large whereas its variance remains relatively small. The problem is that 
both the expectation and  the variance of $Q_x(\go)$ may grow when additional edges are conditioned on being open, and things become difficult
to control when the number of edges that $S$ visits in the block $\bar I_x$ is much larger than $N_0$.
This is the reason why we restrict the use of this method to large values of $m$:
we show that $\bbE\left[ f_x(\go)\ |  \ S \text{ is open }\right]$ is small only for blocks 
for which the number of edges visited by $S$ is not too large.

\medskip

Set
\begin{equation}\label{baras}
 \bar A(S):=\big\{x\in \bar{\mathcal A} \ | \ |S\cap \bar I_x|\le 30 N_0(\log N_0)^{1/4} \big\},
\end{equation}
where $S$ is considered as a set of edges.
As the total number of edges in $S$ is $N\le m N_0 (\log N_0)^{1/4}$ and the $\bar I_x$ are disjoint,
we have
\begin{equation}
  |\bar {\mathcal A}\setminus\bar A(S)|\le \frac{N}{30 N_0(\log N_0)^{1/4}}\le m/30,
\end{equation}
and hence $| \bar A(S)|\ge m/30$.

Thus from \eqref{trtrsi}, and using the fact that all the terms in the product are smaller than one, 
to prove \eqref{crdareq} it is enough to prove that 
for each $x\in \bar A(S)$ we have
\begin{equation}
 \bbE\left[f_{x}(\go)\ | \ S \text{ is open}\right]\le 20000^{-30}.
\end{equation}
Assume in the rest of the proof that $x\in \bar A(S)$.
The definition of $f_x$ gives
\begin{equation}
 \bbE\left[f_{x}(\go)\ | \ S \text{ is open }\right]= 
e^{-K}+(1-e^{-K})\bbP\left[Q_{x}(\go)< e^{K^2}\ | \ S \text{ is open }\right],
\end{equation}
and thus if $K$ is chosen large enough, it is sufficient to prove that  
\begin{equation}\label{vsde}
\bbP\left[Q_{x}(\go)< e^{K^2}\ | \ S \text{ is open }\right]\le 20000^{-31}.
\end{equation}
To obtain such an estimate, it is sufficient to compute the two first moments of $Q_x(\go)$ under the conditioned measure. 
To keep the notation light, we write $\bbP_S$ for $\bbP\left[\ \cdot \ | \ S \text{ is open }\right]$.

\medskip

To show that the first moment is large, one needs to extract a long path of adjacent edges.
Set $S^{(x)}$ to be a path of length $N_0$ defined as follows:
\begin{itemize}
 \item If $x=0$ then 
$(S^{(0)}_n)_{n\in[0,N_0]}:=(S_n)_{n\in[0,N_0]}$ 
\item For all other values of  $x$,
 one sets  $\tau_x$ be the first time that $S$ hits $I_x$ (note that $\tau_x\ge N_0$), and one defines $(S^{(x)}_n)_{n\in [0,N_0]}:=(S_{n+\tau_x-N_0})_{n\in [0,N_0]}$. 
\end{itemize}
Note that $S^{(x)}$ has all its edges in $\bar I_x$.
 
 \medskip

Under the measure $ \bbP_S$, the $\go(e)$  are independent, equal to one 
if $e\in S$, and distributed as Bernoulli variables of parameter $p$ otherwise. 
Thus
\begin{equation}\label{crucru}
  \bbE_S\left[Q_{x}(\go)\right]=\frac{(1-p)}{N_0\sqrt{\log N_0}}
\sumtwo{e, e'\in S\cap \bar I_x}{e\ne e'}\frac{1}{d(e,e')}\ge \frac{(1-p)}{N_0\sqrt{\log N_0}}\sumtwo{e, e'\in S^{(x)}}{e\ne e'}\frac{1}{d(e,e')}.
\end{equation}

Now, for every edge $e\in S^{(x)}$, as the trajectory $S^{(x)}$ cannot move faster than ballistically we have
\begin{equation}
 \sum_{e'\in S^{(x)}\setminus\{e\}}\frac{1}{d(e,e')}\ge  \sum_{n=1}^{N_0-1} \frac{1}{n}
\ge \log N_0.
\end{equation}
Therefore
\begin{equation}\label{cranch}
   \bbE_S\left[Q_{x}(\go)\right]\ge (1-p)\sqrt{\log N_0}.
\end{equation}

\medskip

Let us now bound from above the variance  
$\var_{\bbP_S}\left(Q_{x}(\go)\right)$. 
Most of the terms in the resulting sum appear in the non-conditioned case, so we have to check that the additional terms generated 
by the conditioning only give a small contribution. Indeed,
\begin{multline}
    \var_{\bbP_S}\left[Q_{x}(\go)\right]
=\frac{1}{(1-p)^2 N_0^2\log N_0}\sumtwo{e,e'\in  \bar I_x\setminus{S}}{e'\ne e}
\frac{p^2(1-p)^2}{d(e,e')^2}\\+ \frac{4}{(1-p)^2 N_0^2\log N_0}
\sum_{e\in  \bar I_x\setminus S}p(1-p)^3
\left(\sum_{e'\in S\cap \bar I_x}\frac{1}{d(e',e)}\right)^2.
\end{multline}
The first term in the r.h.s.\ is less than $\var_{\bbP}(Q_x)$ 
(it is the same sum as in Equation \eqref{zrz}, with some missing terms)
and thus is bounded above by $C_1$.
Using the Cauchy-Schwartz inequality, we can bound the second term from above as follows:
\begin{multline}\label{grillo}
\sum_{e\in \bar I_x\setminus S} \left(\sum_{e'\in S\cap \bar I_x}\frac{1}{d(e',e)}\right)^2\le |\bar I_x \cap S|
\sum_{e\in \bar I_x\setminus S} \sum_{e'\in S\cap \bar I_x}\frac{1}{d(e',e)^2}\\
\le
 |\bar I_x \cap S| \sum_{e'\in |\bar I_x\setminus S|} \sum_{\{e\ne  e'\  | \ d(e,e')\le 5N_0\}}\frac{1}{d(e',e)^2}\le 
 C_1  |\bar I_x \cap S|^2 \log N_0.
\end{multline}
Thus  for $x\in \bar {A} (S)$, (recalling \eqref{baras})
\begin{equation}\label{troiqrsi}
     \var_{\bbP_S}\left(Q_{x}(\go)\right)\le C_1(1+3600(1-p)(\log N_0)^{1/2})
\end{equation}
Recall that $N_0:= \exp(\frac{C_2}{(1-p)^2})$, 
and set $$C_2:=4e^{2K^2}.$$ Then
\begin{equation}
    \bbE_S\left[Q_{x}(\go)\right]\ge (1-p)\sqrt{\log N_0}=\sqrt{C_2}= 2e^{K^2}.
\end{equation}
Hence, by the Chebytcheff inequality and \eqref{troiqrsi},
\begin{multline}
 \bbP_S\left(Q_{x}(\go)\le e^{K^2}\right)
 \le  \bbP_S\left(|Q_{x}(\go)- \bbE_S\left[Q_{x}(\go)\right]|\ge e^{K^2}\right)\\
 \le   
\var_{\bbP_S}\left(Q_{x}(\go)\right)e^{-2K^2}\le 8000C_1e^{-K^2},
\end{multline}
which proves \eqref{vsde} if $K$ is large enough, and ends the proof.
\end{proof}

\section{Starting from a supercritical percolation cluster: proof that\\
$p\mapsto \mu_2(p)/(p\mu_2(1))$ is strictly increasing} \label{monnno}\label{newsec}

In the previous section, the proof does not make very much use of the fact that the lattice is $\bbZ^2$, but only the fact that it is two dimensional.
In order to prove that $\mu_2(p)/(p\mu_2(1)$ is strictly increasing, what we have to do is replicate the same proof, but starting with a dilute lattice 
instead of $\bbZ^2$. The reader can note that the proof would adapt to any connected sub-lattice of $\bbZ^2$ or any ``nice" two dimensional lattice.

\medskip

Consider $p<p'$, both in the interval $(p_c,1)$. We couple two percolation environments $\go_p$ and $\go_{p'}$ as we did in Section \ref{tootoo}.
With this coupling, the percolation process with parameter $p$ is obtained by performing percolation with parameter
$q:=p/p'$ on the non-connected lattice
$
\mathcal G_{p'}
$
whose vertices are the same as those of $\bbZ^2$ but whose set of edges is 
$$\mathcal E_p':=\{ e\in E_2 \ | \ \go_{p'}(e)=1\}.$$
As in the $p'=1$ case, it is sufficient to show that for all realizations of $\go_{p'}$

\begin{equation}\label{orrer}
 \bbE[ (Z_N(\go_p))^{1/2} \ | \ \mathcal F_{p'}]\le \left(b^N{q}^N Z_N(\go_{p'})\right)^{1/2}.
\end{equation}
Indeed, using the Borel--Cantelli Lemma, \eqref{orrer} implies that almost surely

\begin{equation}
 \limsup_{N\to \infty} (Z_N(\go_p))^{1/N}<q^N \limsup_{N\to \infty} (Z_N(\go_{p'}))^{1/N},
\end{equation}
and hence
\begin{equation}
 \frac{\mu_2(p)}{p}< \frac{\mu_2(p')}{p'}.
\end{equation}

For the rest of this section we use the structure of the proof  of the case $p'=1$ (the previous section) to prove \eqref{orrer}.
We use the notation $Z^p_N$  for $Z_N(\go_p)$,
and in the proof we have to replace  $\bbZ^2$ by $\mathcal G_{p'}$, $p$ by $q$, $\bbE$ by $\bbE[ \cdot \ | \ \mathcal F_{p'}]$, $s_N$ by $Z^{p'}_N$ and keep in mind that ``open" means
``open for $\go_p$".  We detail only the points where the modifications are not trivial.

\medskip

We set $N_0:=\exp(C_2/(1-q)^2)$ and  after performing our coarse graining
we can  reduce the proof of \eqref{orrer} to proving the inequality
\begin{equation}\label{zact}
 \bbE\left[Z^p_N(\mathcal A)^{\theta}  \ | \ \mathcal F_{p'} \right]\le q^{N\theta} (Z^{p'}_N)^{\theta} 100^{-m}.
\end{equation}
We prove it using Lemma \ref{coaer}; we just need to specify our choice of $f_{\mathcal A}$.

\subsection{The case $m\le N/ (N_0(\log N_0)^{1/4})$}

Set
$$I_x:= \{e\in \mathcal E_{p'}\  | \ r(e)\in N_0x+[0,N_0)^2\}. $$
With this modification the cardinality $| I_{\mathcal A}|$ is not always equal to  $2mN^2_0$; it depends on the realization of $\go_{p'}$. For this reason, for small animals,  we use the following definition for $f_{\mathcal A}$ (instead of \eqref{smallf})
\begin{equation}
f_{\mathcal A}:=\frac{\gl^{\#\{ \text{ open edges for $\go_{p}$ in $I_{\mathcal A}$ }\}}}{(1-q(1-\gl))^{| I_{\mathcal A}|}}
\end{equation}
As $| I_{\mathcal A}|\le 2mN^2_0$, \eqref{stimac} remains valid with this change of definition. Equation \eqref{roselmc} is replaced by
\begin{equation}
\bbE[ f_{\mathcal A} Z^p_N(\mathcal A) \ | \ \mathcal F_{p'}]=Z^{p'}_N(\mathcal A)q^N\left(\frac{\gl}{1-q(1-\gl)}\right)^N,
\end{equation}
which allows us to conclude that
\begin{equation}
\bbE[ f_{\mathcal A} Z^p_N(\mathcal A) \ | \ \mathcal F_{p'}]\le Z^{p'}_Nq^N 10^{-5m},
\end{equation}
and to derive \eqref{zact} from Lemma \ref{coaer}.

\subsection{The case of $m> N/ (N_0(\log N_0)^{1/4})$}

For large animals, we have to check that we can still control the variance and expectation of $Q_x$ (defined with the modified version of $\bar I_x$ and 
$q$ instead of $p$)
when we start with a dilute lattice.
The fact that $I_x$ contains fewer edges only makes the variance of $Q_x$ smaller. Indeed
the sums are made on a subset of the indices. Hence \eqref{zrz} and Lemma \ref{coaer2} are still valid (even if the $f_x$ are independent but not identically distributed).

\medskip

Lemma \ref{crdar} is replaced by 

\begin{lemma}

For every $S\in \mathcal S_N(\mathcal A)$ which is open for $\go_{p'}$, we have
 \begin{equation}\label{crdareq2}
 \bbE\left[f_{\mathcal A}(\go_p)\ind_{\{ S \text{ is open for $\go_p$}\}}\  | \ \mathcal F_{p'}\right]\le 
q^ N 20000^{-m}.
\end{equation}
\end{lemma}

Equations \eqref{cranch} and \eqref{troiqrsi} are still valid if 
$\bbP_S$ is replaced by ``$\bbP[ \cdot | \mathcal F_{p'}]$ conditioned to $S$ being open for $\go_p$": for \eqref{cranch} the estimate remains the same, and for
 \eqref{troiqrsi} we just have a sum on a smaller family of edges. Hence the proof is exactly as for Lemma~\ref{crdar}.

\section{The Self-avoiding walk in a Random Potential}\label{omega}

Percolation is just one type of random potential in which one can make $S$ evolve but there are many other possibilities.
We introduce in this section a general self-avoiding walk in a random potential and show that at every temperature the partition function grows exponentially slower than its expectation. 
Let $\go(x)$ be a collection of IID random variables of zero mean and unit variance, indexed  by sites of $\bbZ^2$ (let $\bbP$ denote their joint law)
satisfying
$$e^{\gl(\gb)}:=\bbE\left[e^{\gb \go(0)}\right]<\infty$$ 
for all $\gb\in \bbR$.

 The Self-avoiding Walk in a Random Potential is the stochastic process whose law is given by 
 the probability measure on $\mathcal S_N$ for which each path $S$ has a probability proportional to
\begin{equation}
 \Pi_{\go}(S)=\exp\left(\gb \sum_{n=0}^N\go(S_n)\right).
\end{equation}
Physically, $(-\go)$ corresponds to an energy attached to each site, and $\gb$ is the inverse temperature.
We are interested in the growth rate of $Z_N$ the partition function of this model defined  by
\begin{equation}
 Z_N(\gb,\go):=\sum_{S\in \mathcal S_N} \Pi_{\go} (S).
\end{equation}

\begin{theorem}
For any $\gb\ne 0$, we have
\begin{equation}
\limsup_{N\to \infty}  (Z_N)^{1/N} < \limsup_{N\to \infty} \bbE\left[ Z_N \right]^{1/N}= \mu_2 e^{\gl(\gb)}.
\end{equation}
\end{theorem}

As in the percolation case, we remark that it is sufficient to show that 
\begin{equation}\label{partifunk}
\bbE\left[\left(Z_N(\gb,\go)(s_N)^{-1} e^{-N\gl(\gb)}\right)^{1/2}\right],
\end{equation}
decays exponentially fast in $N$.
We prove \eqref{partifunk} with the same line of proof as Theorem~\ref{superteo}. We do not reproduce the parts of the proof that are identical to the percolation case.
Without loss of generality we assume $\gb>0$ in the proof.
The first thing we will do is to restrict ourselves to the case of small $\gb$.
We do so by using the FKG inequality, similarly to what is done in \cite{cf:CY}. Afterwards we use the animal decomposition, and
bound the contribution of each animal by using change of measure.

\subsection{Restriction to small $\gb$}

We show in this section that when $N$ is fixed 
the quantity \eqref{partifunk} is a non-increasing function of $\gb$. For this, we follow \cite[Lemma 3.3 (b)]{cf:CY}, and show that the derivative in $\gb$ is non-positive: 

\begin{multline}
\partial_\gb\bbE\left[\left(Z_N(\gb,\go) e^{-N\gl(\gb)}\right)^{1/2}\right]
\\=\frac{1}{2}\sum_{S\in \mathcal S_N} \bbE\left[\left(Z_N(\gb,\go) e^{-N\gl(\gb)}\right)^{-1/2}
\left[\left(\sum_{n=0}^N\go(S_n)\right)-N\gl'(\gb)\right]\Pi_{\go} (S)e^{-N\gl(\gb)}\right].
\end{multline}
For a fixed $S$, the measure $\bbP_S$ defined by
\begin{equation}\label{bbps}
\bbP_S(\dd \go):=\Pi_{\go} (S)e^{-N\gl(\gb)}\bbP(\dd \go)
\end{equation}
 is a product probability measure, and hence satisfies the FKG inequality \cite[p. 78]{cf:Lig}.
As $\left(Z_N(\gb,\go) e^{-N\gl(\gb)}\right)^{-1/2}$ and $\left[\left(\sum_{n=0}^N\go(S_n)\right)-N\gl'(\gb)\right]$ are respectively decreasing and increasing functions of $\go$, one has
\begin{multline}
\bbE\left[\left(Z_N(\gb,\go) e^{-N\gl(\gb)}\right)^{-1/2}
\left[\left(\sum_{n=0}^N\go(S_n)\right)-N\gl'(\gb)\right]\Pi_{\go} (S)e^{-N\gl(\gb)}\right] \\
\le \bbE_S\left[\left(Z_N(\gb,\go) e^{-N\gl(\gb)}\right)^{-1/2}\right]\bbE_S\left[
\left(\sum_{n=0}^N\go(S_n)\right)-N\gl'(\gb)\right]=0,
\end{multline}
where the last equality is due to the fact that $\bbE_S[\go(x)]=\gl'(\gb)\ind_{x\in S}$. 
In what follows we will always consider that $\gb\le \gb_0$ is small enough (how small will depend on the law of $\go$).

\subsection{Coarse graining}
Fix $N_0:=\exp(C_2\gb^{-4})$. Similarly to what is done in Section~\ref{cganim}, we can reduce \eqref{partifunk} to the proof of a statement analogous to Proposition \ref{fundamental} controling the contribution of each animal.
 We want to prove that
for every $\mathcal A\in \mathfrak{A}_m$, we have
\begin{equation}\label{zactos2}
 \bbE\left[Z_N(\mathcal A)^{1/2}\right]\le  s^{1/2}_N e^{N\gl(\gb)/2} 100^{-m},
\end{equation}
where 
$$Z_N(\mathcal A):=\sum_{S\in \mathcal S_N(\mathcal A)} \Pi_{\go} (S).$$
We will prove it using Lemma \ref{coaer} and appropriate functions $f_{\cA}$, 
separating the cases  $m\le N/[N_0(\log N_0)^{1/4}]$ and $m> N/[N_0(\log N_0)^{1/4}]$.

\subsection{The case $m\le N/[N_0(\log N_0)^{1/4}]$}

In the first case we set
 \begin{equation}
  f_{\mathcal A}(\go):= \exp\left(-\left(\sum_{x\in I_{\mathcal A}} \delta\go_x\right)-m N^2_0\gl(-\gd))\right),
 \end{equation}
 where the definition of $I_{\mathcal A}$ and $I_x$ has been adapted so that they are sets of points rather than sets of edge ($I_x:=\bbZ^2\cap [0,N_0)^2$ contains $N^2_0$ points), 
and $\gd=\gd_{N_0}:=(1/N_0)$. This corresponds to an exponential tilt of the $\go$ inside $I_{\mathcal A}$.

\medskip

Then, provided that $N_0$ is large enough (hence for $\gb$ small enough) one has
\begin{equation}\label{reon}
\bbE\left[f_{\mathcal A}^{-1}\right]=\exp\left(mN_0^2 \left[\gl(-\gd)+\gl(\gd)\right]\right)\le \exp(mN_0^2 2 \gd^2)=e^{2m},
\end{equation}
where we used that $\gl(\gd)\sim_{0} \gd^2/2$ from the assumption on the variance of $\go$.

\medskip

On the other hand,
\begin{multline}\label{rearf}
 \bbE\left[Z_N(\mathcal A)f_{\mathcal A}\right]=|\mathcal S_N(\mathcal A)| \exp(N[\gl(\gb-\delta)-\gl(-\delta)])\\
 \le s_N e^{N\gl(\gb)} \exp\left(-m N_0(\log N_0)^{1/4} [ \gl(\gb)-\gl(\gb-\delta)+\gl(-\delta)]\right),
 \end{multline}
where the last inequality uses our assumption $m\le N/[N_0(\log N_0)^{1/4}]$ and the facts $[ \gl(\gb)-\gl(\gb-\delta)-\gl(-\delta)]>0$, $|\mathcal S_N(\mathcal A)|\le s_N$. 
Moreover by the mean value theorem (used twice), there exists $\gd_1\in [0,\delta]$ and $\gb_1\in[-\delta_1,\gb-\delta_1]$ such that
\begin{equation}
\gl(\gb)-\gl(\gb-\delta)+\gl(-\delta)-\gl(0)=\gd\left(\gl'(\gb-\delta_1)-\gl'(-\delta_1)\right)=\gd\gb \gl''(\gb_1).
\end{equation}
As $\gl''(s)$ tends to one at zero, $ \gl''(\gb_1)\ge 1/2$ if $\gb$ is small enough; hence

\begin{multline}\label{reac}
\exp\left(-m N_0(\log N_0)^{1/4} [ \gl(\gb)-\gl(\gb-\delta)+\gl(-\delta)]\right)\\
\le e^{-mN_0(\log N_0)^{1/4}\delta\gb/2}\le e^{-mC_2^{1/4}/2}\le 10^{-5m}.
\end{multline}
and one can conclude that \eqref{zactos2} holds by combining \eqref{reon}, \eqref{rearf}, \eqref{reac} and Lemma~\ref{coaer}.

\subsection{Case $m> N/[N_0(\log N_0)^{1/4}]$}

Let us now move to the case of large $m$.
We introduce a quadratic form 
\begin{equation}
Q_x(\go):=\frac{1}{N_0\sqrt{\log N_0}}\sumtwo{z,z'\in \bar I_x}{z\ne z'}\frac{1}{d(z,z')}\go(z)\go(z'),
\end{equation}
where $d$ denote the Euclidean distance
and $f_x$ and $f_{\mathcal A}$ are defined as in \eqref{fxxx} and \eqref{faaa}.
Then there is no problem in proving the equivalent of Lemma \ref{coaer2}, just by controlling the variance of $Q_x$ as follows
\begin{equation}
\bbE[Q^2_x(\go)]\le \frac{1}{N_0^2\log N_0}\sum_{z\in \bar I_x}\sum_{\{z' | d(z,z')\le 5N_0\}} \frac{1}{d^2(z,z')}\le C_1.
\end{equation}

Instead of Lemma \ref{crdar}, we must then prove that for any $S$ in $\mathcal S_N(\mathcal A)$ and $x\in \bar {\mathcal A}$ such that 
$|S\cap \bar I_x|\le 30 N_0(\log N_0)^{1/4}$ (recalling \eqref{bbps}) we have
\begin{equation}\label{cbpf}
\bbE[\Pi_\go(S)f_{\mathcal A}(\go)]=e^{N\gl(\gb)}\bbE_S\left[f_{\mathcal A}(\go)\right]\le e^{N\gl(\gb)}20000^{-m},
\end{equation}
which is proved by controlling the mean and variance of $Q_x(\go)$ under $\bbP_S$.
The computations are almost the same as for the proof of Lemma \ref{crdar}. One notices that

\begin{equation}
\bbE_S[\go(z)]=\gl'(\gb)\ind_{z\in S} \text{ and } \var_{\bbP_S}[\go(z)]=\gl''(\gb)\ind_{z\in S}+\ind_{z\notin S}.
\end{equation}
In what follows, we consider $\gb$ small enough so that 
$$\gb/2\le \gl'(\gb)\le 2\gb \quad \text{ and } \quad  \gl''(\gb)+\gl'(\gb)^2\le 2$$ 
(recall that $\gl''(0)=1$ from the unit variance assumption). 

One obtains in the same fashion as \eqref{cranch} that
\begin{equation}\label{esperan}
\bbE_S[Q_x(\go)]=\frac{\gl'(\gb)^2}{N_0\sqrt{N_0}}\sumtwo{z,z'\in S\cap \bar I_x}{z\ne z'} \frac{1}{d(z,z')}\ge \frac{\gb^2}{4}\sqrt{\log N_0}.
\end{equation}
Now, we have to control the variance under $\bbP_S$.
\begin{multline}
\var_{\bbP_S}[Q_x(\go)]=\frac{1}{N_0^2\log N_0}\sumtwo{z_1,z_2\in \bar I_x}{z_1\ne z_2}\sumtwo{z_3,z_4\in \bar I_x}{z_3\ne z_4}\frac{1}{d(z_1,z_2)d(z_3,z_4)}\cov_{\bbP_S}(\go_{z_1}\go_{z_2}, \go_{z_3}\go_{z_4})\\
\le \frac{2}{N_0^2\log N_0}\sumtwo{z,z'\in \bar I_x}{z\ne z'}\frac{1}{d(z,z')^2}\bbE_{\bbP_S}[\go^2_{z}\go^2_{z'}]
\\
+  \frac{4\gl'(\gb)^2}{N_0^2\log N_0}\sumtwo{z\in \bar I_x,\ z',z"\in (S\cap \bar I_x)}{z\notin\{z', z''\}}\frac{1}{d(z,z')d(z,z'')}\bbE_S[\go^2_z].
\end{multline}
The two terms on the r.h.s.\ correspond to the two cases where the covariance \\
$\cov_{\bbP_S}(\go_{z_1}\go_{z_2}, \go_{z_3}\go_{z_4})$ is non-zero:
either $z_1=z_3$, $z_2=z_4$ (or $z_1=z_4$, $z_2=z_3$) which is counted in the first term, 
or $z_1=z_3$ and $z_2\ne z_4$ but $z_2$ and $z_4$ belong to $S$ (and three other cases obtained by permutation of the indices).
In each case we have bounded the covariance from above by neglecting to subtract the product of expectations $\bbE_S[\go_{z_1}\go_{z_2}]\bbE_S[ \go_{z_3}\go_{z_4}]$, which is always positive.

\medskip

Note that, with our assumptions, in the sum above we have $\bbE_{\bbP_S}[\go^2_{z}\go^2_{z'}]\le 4$ and $\bbE_S[\go^2_z]\le 2$.
Hence, the first term is smaller than
\begin{equation}
 \frac{8}{N_0^2\log N_0}\sum_{z\in \bar I_x}\sum_{\{z'\ne z\ |\ |z-z'|\le 5N_0\}}\frac{1}{d(z,z')^2}\le 8C_1.
\end{equation}
For the second term, as in \eqref{grillo}  one gets
\begin{equation}
 \sumtwo{z\in \bar I_x,\ z',z"\in (S\cap \bar I_x)}{z\notin\{z', z''\}}\frac{1}{d(z,z')d(z,z'')}= \sum_{z\in \bar I_x} \left(\sumtwo{z'\in S\cap \bar I_x}{z'\ne z}\frac{1}{d(z,z')}\right)^2\le |S\cap \bar I_x|^2 C_1\log N_0,
 \end{equation}
so that the second term is smaller than
$32 C_1\gb^2 (|S\cap \bar I_x|/N_0)^2$.

\medskip

Hence, when $|S\cap \bar I_x|\le 30 N_0 (\log N_0)^{1/4}$ we have
\begin{equation}\label{varian}
\var_{\bbP_S}[Q_x(\go)]=C_1(8+28800 \gb^2 (\log N_0)^{1/2}).
\end{equation}
Equations \eqref{esperan} and \eqref{varian} allow us to prove \eqref{cbpf}, just as \eqref{cranch} and \eqref{troiqrsi} are used to prove Lemma \ref{crdar}. 

\section{Some other models to which the proof can adapt}\label{export}

Note that our proof, though technical, did not use many specifics of the model.
The key point that makes the proof work is that we can find $Q_x$ for which the variance is bounded but whose expectation under $\bbE_S$ diverges with $N_0$.
This is where the crucial fact that the lattice is $2$-dimensional is used.
For this reason, our result extends readily to any kind of two-dimensional lattice (e.g.\ the triangular lattice, the honeycomb lattice, lattices with spread-out connections).
Moreover, the proof would also work with only minor modification for a large variety of $2$-dimensional models, an example of which was given in the last section.
Without trying to describe a meta-model that would include all of these models, we give here two further examples that could be of interest.

\subsection{Site Percolation}
We consider the equivalent of the model studied in the core  of this paper, but with the disorder $(\go(x))_{x\in \bbZ^2}$ lying on the sites of $\bbZ^2$
rather than on the edges. One says that a self-avoiding path is open if all the sites visited by the path are open. One can readily check  that 
 one can prove the existence of the quenched connective constant (with $\limsup$) and the fact that it 
differs from the annealed one for every $p$ using exactly the same arguments. One can furthermore adapt Section \ref{monnno} to show that  the ratio of the two connective constants is an increasing function of $p$.

\subsection{Lattice trees/Lattice animals on a dilute network}

A lattice tree of size $N$ on~$\bbZ^2$ is a finite connected subgraph of $\bbZ^2$ with $N$ vertices and no cycles.
Let $\mathcal T_N$ denote the number of (unlabeled) lattice trees of size $N$ containing the origin in $\bbZ^2$ and $t_n=|\mathcal T_N|$.
It is known (see e.g.\ \cite{cf:Kl}, where lattice trees are studied as a model for branching polymers) that there exists a constant $\gl$ such that
\begin{equation}
 \lim_{N\to \infty} (t_N)^{\frac 1 N}= \gl.
\end{equation}
The method presented in Section \ref{pprroof} also allows us to give an upper bound on the number of lattice trees present on an infinite percolation.
We say that a lattice tree is open if all of its the edges are open.
Given a realization $(\go(e))_{e\in E_2}$  of the edge dilution process, one defines

\begin{equation}
 Z_N:=\sum_{T\in \mathcal T_N} \ind_{\{T \text{ is open }\}},
\end{equation}
letting $Z_N(x)$ be the analogous sum for lattice trees containing $x$.
It can be shown, as we have done for self-avoiding paths, that the upper-growth rate
\begin{equation}
\limsup_{N\to\infty} (Z_{N}(x))^{1/N}
\end{equation}
is constant in $x$ and non-random on the infinite percolation cluster. 
Using exactly the same proof as in Section \ref{pprroof}, one can further prove

\begin{theorem}
For any $p\in(1/2,1)$ one has  $\bbP_p$ a.s.\ for all point $x\in C$, we have
 \begin{equation}
  \limsup_{N\to\infty} (Z_{N}(x))^{1/N}<   \limsup_{N\to\infty} (\bbE_p\left[Z_{N}(x)\right])^{1/N}=p\gl.
 \end{equation}
\end{theorem}

\begin{proof}
The only point that needs to be explained in the adaptation of the proof is how one chooses the $S^{(x)}$ appearing in 
\eqref{crucru}. For $x=0$, we choose arbitrarily a path of length $N_0$ moving away from the root (there has to be at least one if
$N\ge N_0^2$). For the other values of $x$, we take $S_0^{(x)}$ to be a point of the tree that lies in $I_x$, and 
$S^{(x)}$ to be first $N_0$ steps on the paths from this point towards the root, i.e.\ the origin. 
\end{proof}

A similar result could be stated for lattice animals on the supercritical percolation cluster. One could also consider trees or lattice animals in a random potential that is not percolation, and prove a result similar to the one of Section \ref{omega}.

\medskip

{\bf Acknowledgement:} The author is grateful to an anonymous referee for his careful examination of the paper,  and numerous suggestions to improve its quality. He also would like thank to Ben Smith for his help on the text.
 This work was initiated during the authors stay in the Instituto de Matematica Pura e Applicada, he acknowledges kind hospitality and the support of CNPq.

\end{document}